\documentclass{article}    

\usepackage[all]{xy}
\usepackage{latexsym,amssymb,amsfonts,amsmath,bm,amsthm}
\usepackage{proof}
\newcommand{\pullbackcorner}[1][dr]{\save*!/#1-1.2pc/#1:(-1,1)@^{|-}\restore}
 \newtheorem{definition}{Definition}
  \newtheorem{theorem}[definition]{Theorem}
  \newtheorem*{theorem*}{Theorem}
  \newtheorem{lemma}[definition]{Lemma}

  \theoremstyle{remark}

\def\sub{\subseteq}
\def\cov{\triangleleft}
\def\fish{\ltimes}

\def\P{\mathcal{P}}
\def\S{\mathcal{S}}

\title{Factorizing the $\mathbf{Top}$--$\mathbf{Loc}$ adjunction through positive topologies}

\author{Francesco Ciraulo\\ {\footnotesize Department of Mathematics, University of Padova}\\ {\footnotesize Via Trieste 63, 35121 Padova, Italy}\\ \phantom{ }\\ Tatsuji Kawai\\ {\footnotesize Japan Advanced Institute of Science and Technology}\\ 
{\footnotesize 1-1 Asahidai, Nomi, Ishikawa 923-1292, Japan}\\ \phantom{ }\\ Samuele Maschio\\ {\footnotesize Department of Mathematics, University of Padova}\\ {\footnotesize Via Trieste 63, 35121 Padova, Italy}}

\date{}

\begin{document}

\maketitle

\begin{abstract}
We characterize the category of Sambin's positive topologies as a fibration over the category of locales {\bf Loc}. The fibration is obtained by applying the Grothendieck construction to a doctrine over {\bf Loc}. We then construct an adjunction between the category of positive topologies and that of topological spaces {\bf Top}, and show that the well-known adjunction between {\bf Top} and {\bf Loc} factors through the newly constructed adjunction.\\


\end{abstract}

%\begin{document}

%\title{Factorizing the $\mathbf{Top}$--$\mathbf{Loc}$ adjunction through positive topologies
%\thanks{Grants or other notes %about the article that should go on the front page should be

%\author{Francesco Ciraulo\and Tatsuji Kawai\and Samuele Maschio}

%\authorrunning{Ciraulo-Kawai-Maschio} % if too long for running head

%\institute{F. Ciraulo \at
%              Department of Mathematics, University of Padova, Via Trieste 63, 35121 Padova, Italy\\
              %Tel.: +39-049-827-1357\\
%              \email{ciraulo@math.unipd.it}           %  \\
%           T. Kawai \at
%              School of Information Science, Japan Advanced Institute of Science and Technology, 1-1 Asahidai, Nomi, Ishikawa 923-1292, Japan\\
%\email{tatsuji.kawai@jaist.ac.jp}
						\and
%						S. Maschio \at
%						Department of Mathematics, University of Padova, Via Trieste 63, 35121 Padova, Italy\\ 
%              \email{maschio@math.unipd.it} 	
%}


%\journalname{Archive for Mathematical Logic}

%\maketitle

%\begin{abstract}
%We characterize the category of Sambin's positive topologies as a fibration over the category of locales {\bf Loc}. The fibration is obtained by applying the Grothendieck construction to a doctrine over {\bf Loc}. We then construct an adjunction between the category of positive topologies and that of topological spaces {\bf Top}, and show that the well-known adjunction between {\bf Top} and {\bf Loc} factors through the newly constructed adjunction.
%\keywords{Grothendieck constructions\and suplattices\and locales\and formal topologies}
%\PACS{PACS code1 \and PACS code2 \and more}
%\subclass{06D22\and 18B30\and 03G30}
%\end{abstract}

\section{Introduction}\label{prequel}

Positive topologies are introduced in \cite{bp} (see also \cite{CS}) as a natural structure for developing pointfree constructive topology. The category {\bf PTop} of positive topologies can be regarded as a natural extension of the category {\bf Loc} of locales; actually {\bf Loc} is a reflective subcategory of {\bf PTop}. In a predicative setting, the role of a locale is played by a formal cover $(S,\triangleleft)$ which can be read as a presentation of a frame by generators and relations. A positive topology is then a formal cover 
endowed with a positivity relation, that is a relation $\ltimes$ between $S$ and $\mathcal{P}(S)$ such that for every $a\in S$ and  $U,V\subseteq S$
\begin{enumerate}
\item $a\ltimes U$ $\Longrightarrow$ $a\in U$;
\item $a\ltimes U\wedge (\forall b\in S)(b\ltimes U\rightarrow b\in V)$ $\Longrightarrow$ $a\ltimes V$;
\item $a\triangleleft U\wedge a\ltimes V$ $\Longrightarrow$ $(\exists b\in U)(b\ltimes V)$.
\end{enumerate}
The motivating example of a positive topology is built from a topological space, in such a way to keep the information about its closed subsets (classically, all such information is already encoded by the opens); see Section \ref{sec:CanPos}.

In \cite{CV} the first author and S.\ Vickers characterize positive topologies as locales endowed with a family of suplattice homomorphisms. Here we show that this characterization can be organized into a fibration arising from a doctrine via the so-called Grothendieck construction (see, e.g.\ \cite{BJ}). 

We will then use this representation to give an adjunction between the category {\bf Top} of topological spaces and {\bf PTop}; in particular, the notion of sobriety provided by this adjunction coincides with the one introduced in \cite{bp},  which is known to be constructively weaker than the usual one \cite{aczel-fox}.
Moreover, the usual {\bf Top}--{\bf Loc} adjunction can be factorized as the composition of the {\bf Top}--{\bf PTop} adjunction above and the reflection {\bf PTop}--{\bf Loc}.

As a by-product, we get the completeness and cocompleteness of the category {\bf PTop} (and of the wider category {\bf BTop}). This completes the picture in~\cite{IK}, where the pointwise counterparts of {\bf BTop} and {\bf PTop} were shown to be complete and cocomplete.

Our foundational framework is intuitionistic and impredicative, like that provided by the internal language of a topos. We use the term ``constructive'' in this sense.
% or by an impredicative extension of the extensional level of the Minimalist Foundation \cite{MinF1,MinF2}. 

%In the last section we will discharge impredicativity and we will discuss what happens in the predicative case.

\section{Basic topologies and positive topologies}\label{section bp}

According to the usual definition, a {\bf suplattice} (aka \emph{complete join semilattice}) is a poset $(L,\leq)$ with all joins, that is $\bigvee X$ exists for all subsets $X\subseteq L$.\footnote{In particular, $L$ has the least element 0, namely the empty join. (Actually $L$ has the top element 1 as well, although this should be understood as the empty meet.)}\\ A map $f:L\to M$ \emph{preserves joins} if
$$f\big(\bigvee_{i\in I}x_i\big) = \bigvee_{i\in I}f(x_i)$$
for every family $(x_i)_{i\in I}$ in $L$.
Suplattices and join-preserving maps form a category $\mathbf{SL}$. We hence refer to join-preserving maps between suplattices as suplattice homomorphisms. 

A \emph{base} for $L$ is a subset $S\sub L$ such that $p$ = $\bigvee\{a\in S\ |\ a\leq p\}$ for all $p \in L$. For instance, the powerset $\P(S)$ of a set $S$ is a suplattice (with respect to union); a base for $\P(S)$ is given by all singletons.\footnote{Incidentally, note that $\P(S)$ is the free suplattice over the set $S$.} Given a base $S$, let $\cov\sub S\times\P(S)$ be the relation defined as $a\cov U$ iff $a\leq\bigvee U$. It is easy to check that $\cov$ satisfies the following properties
\begin{enumerate}
\item $a\in U\ \Longrightarrow\  a\cov U$
\item $a\cov U\ \land\ (\forall u\in U)(u\cov V)\ \Longrightarrow\ a\cov V$
\end{enumerate}
for every $a\in S$ and $U,V\sub S$. A pair $(S,\cov)$ satisfying 1 and 2 above is called a {\bf basic cover}.
A basic cover has to be understood as a presentation of a suplattice by generators and relations. Indeed, any cover induces an equivalence relation $=_\cov$ on $\P(S)$ where $U=_\cov V$ is $(\forall u\in U)(u\cov V)\land (\forall v\in V)(v\cov U)$. The quotient collection $\P(S)_{/_{=_\cov}}$ is a suplattice (with a base indexed by $S$) where joins $\bigvee_i[U_i]$ can be computed as $[\bigcup_i U_i]$. To complete the picture, one should note that the cover induced by a suplattice $L$ presents a suplattice which is isomorphic to $L$ itself.

Two basic covers $\S_1=(S_1,\cov_1)$ and $\S_2=(S_2,\cov_2)$ are isomorphic if they induce isomorphic suplattices. More generally we say that a morphism from $\S_1$ and $\S_2$ is a suplattice homomorphism from $\P(S_2)/=_{\cov_2}$ and $\P(S_1)/=_{\cov_1}$.\footnote{Contravariance is chosen to match the direction of locales.} This corresponds to having a relation $s\sub S_1\times S_2$ which {\bf respects the covers} in the following sense:
$$\textrm{if }a\,s\,b\textrm{ and } b\cov_2 V\textrm{, then }a\cov_1 s^-V$$ 
where $s^- V=\{x\in S_1\ |\ (\exists v\in V)(x\,s\,v)\}$. Actually, the same homomorphism corresponds to several relations which we want to consider equivalent; explicitly, two relations $s$ and $s'$ are equivalent if $s^-V=_{\cov_1}s'^-V$ for all $V\sub S_2$.

Basic covers and their morphisms form a category which is dual to the category ${\bf SL}$ of suplattices, that is, equivalent to $\mathbf{SL}^{op}$.

\subsection{Basic topologies}\label{sec:BTop}

A {\bf basic topology} is a triple $(S,\triangleleft,\ltimes)$ where $(S,\triangleleft)$ is a basic cover and $\ltimes$ is a relation between $S$ and $\mathcal{P}(S)$ such that
\begin{enumerate}
\item $a\ltimes U$ $\Longrightarrow$ $a\in U$;
\item $a\ltimes U\wedge (\forall b\in S)(b\ltimes U\rightarrow b\in V)$ $\Longrightarrow$ $a\ltimes V$;
\item $a\triangleleft U\wedge a\ltimes V$ $\Longrightarrow$ $(\exists b\in U)(b\ltimes V)$.
\end{enumerate}
The relation $\ltimes$ is called a {\bf positivity relation} on $(S, \triangleleft)$.  Thus, a basic topology can be regarded as a suplattice together with the extra structure specified by a positivity relation.

The powerset $\Omega=\mathcal{P}(1)$ of a singleton can be identified with the algebra of propositions up to logical equivalence. 
The last condition in the definition above says that the map
$$\begin{array}{ccccc}
\varphi_Z & : & \P(S)/=_{\cov} & \longrightarrow & \Omega\phantom{\ \ } \\
 & & {[U]} & \longmapsto & U\between Z\ \footnotemark
\end{array}$$
\footnotetext{For $U,V\subseteq S$, we use Sambin's ``overlap" symbol $U\between V$ to mean that $U\cap V$ is inhabited. Clearly $U\between V$ implies $U\cap V\neq\emptyset$. The converse is equivalent to the law of excluded middle, as it is clear by considering the case of $\Omega$. In that case, $p\between q$ means $p=q=1$, and  so $p\between p$ is just $p=1$. On the contrary, $p\cap p\neq 0$ is $p\neq 0$, that is $(\neg\neg p)=1$.}%
is well-defined if $Z$ is of the form $\{a\in S\ |\ a\fish V\}$, in which case, $\varphi_Z$ is a suplattice homomorphism. In other words, each $\{a\in S\ |\ a\fish V\}$ gives a suplattice interpretation of $(S,\triangleleft)$ into the set $\Omega$ of truth values. Given any positivity relation $\fish$ on $(S,\cov)$, the collection of all such $\varphi_Z$ forms a sub-suplattice of $\mathbf{SL}({\P(S)/\! =_{\cov}}\,,\,\Omega)$. Ciraulo and Vickers \cite[Theorem 2.3]{CV} have shown that there is a bijective correspondence between positivity relations on $(S,\triangleleft)$ and sub-suplattices of $\mathbf{SL}({\P(S)/\! =_{\cov}}\,,\,\Omega)$. Thus, a basic topology can be identified with a pair $(L, \Phi)$ of a suplattice  $L$ and a sub-suplattice $\Phi$ of the collection $\mathbf{SL}(L,\Omega)$ of suplattice homomorphisms from $L$ to $\Omega$.

Let $\S_1 = (S_1, \triangleleft_1, \ltimes_1)$ and $\mathcal{S}_2 = (S_2, \triangleleft_2, \ltimes_2)$ be basic topologies, and $(L_1, \Phi_1)$ and $(L_2,\Phi_2)$ be the corresponding suplattices together with sub-suplattices of suplattice homomorphisms into $\Omega$. According to \cite{bp}, a morphism between basic topologies $\mathcal{S}_1$ and $\mathcal{S}_2$ is a morphism $s$ between $(S_1,\cov_1)$ and $(S_2,\cov_2)$ satisfying the following additional condition
$$\textrm{if }a\,s\,b\textrm{ and }a\fish_1 U\textrm{, then } b\fish_2 s\,U$$
for all $a\in S_1$, $b\in S_2$ and $U\sub S_1$, where $s\, U=\{y\in S_2 \mid (\exists\,u\in U)(u\,s\,y)\}$. This corresponds to having a suplattice homomorphism $f : L_2 \to L_1$ such that $\Phi_1 \circ f \subseteq \Phi_2$, where $\Phi_1 \circ f = \left\{ \varphi \circ f \mid \varphi \in \Phi_1 \right\}$; in other words
$$\textrm{if }L_1\stackrel{\varphi}{\longrightarrow}\Omega\textrm{ belongs to }\Phi_1\textrm{, then }L_2\stackrel{f}{\longrightarrow}L_1\stackrel{\varphi}{\longrightarrow}\Omega\textrm{ belongs to }\Phi_2$$
(see \cite[Proposition 2.9]{CV}).

Let $\mathbf{BTop}$ be the category whose objects are pairs $(L,\Phi)$ of a suplattice $L$ and a sub-suplattice $\Phi$ of $\mathbf{SL}(L,\Omega)$ and whose arrows $f : (L_1,\Phi_1) \to (L_2,\Phi_2)$ are suplattice homomorphisms $f : L_2 \to L_1$ such that $\Phi_1 \circ f \subseteq \Phi_2$. Apart from the impredicativity involved, $\mathbf{BTop}$ is equivalent to the category of basic topologies in \cite{bp}.

\subsection{Positive topologies}\label{sec:PTop}

A {\bf positive topology} is a basic topology $(S,\triangleleft,\ltimes)$ such that the underlying basic cover $(S,\triangleleft)$ is a formal cover \cite{cms13} (sometimes called formal topology). This means that the suplattice presented by $(S,\triangleleft)$ is a {\bf frame}, that is, binary meets distribute over arbitrary joins.

By a similar observation as we have made for basic topology in Section \ref{sec:BTop}, a positive topology can be identified with a pair $(L, \Phi)$ where $L$ is a frame and $\Phi$ is a sub-suplattice of $\mathbf{SL}(L,\Omega)$. A morphism between such pairs $(L, \Phi)$ and $(M, \Psi)$ is a frame homomorphism $f : M \to L$ such that  $\Phi \circ f \subseteq \Psi$, which corresponds to a formal map between positive topologies as described in \cite{bp}.

Let $\mathbf{PTop}$ be the subcategory of $\mathbf{BTop}$ consisting of objects whose underlying suplattice is a frame and arrows which are frame homomorphisms between underlying frames.  The category $\mathbf{PTop}$ is thus essentially equivalent to that of positive topologies in \cite{bp}.

\section{A categorical characterization of {\bf BTop} and {\bf PTop}}
%\subsection{Suplattices}

In this section, we are going to give a categorical characterization of {\bf BTop} and {\bf PTop} in terms of Grothendieck constructions over two doctrines on suplattices and locales, respectively.

If $X$ is a set and $L$ is (the carrier of) a suplattice, then the collection of maps $\mathbf{Set}(X,L)$ has a natural suplattice structure where joins are computed pointwise, that is,
$$\big({\bigvee_{i\in I}}\varphi_{i}\big)(x):={\bigvee_{i\in I}}\big(\varphi_{i}(x)\big)\ .$$ When $X$ has a suplattice structure, then $\mathbf{SL}(X,L)$ is a suplattice as well, actually a sub-suplattice of $\mathbf{Set}(X,L)$. 
%Indeed $(\bigvee_i\varphi_i)(\bigvee_j x_j)$ = $\bigvee_i(\varphi_i(\bigvee_jx_j))$ = $\bigvee_{i,j}\varphi_i(x_j)$ = $\bigvee_j(\bigvee_i(\varphi_i(x_j))$ = $\bigvee_j((\bigvee_i\varphi_i)(x_j))$ for each set-indexed family $(\varphi_{i})$ of lattice homomorphisms and each set-indexed family $(x_j)$ of elements of $X$. 

\subsection{A doctrine on suplattices} \label{sec:DoctSL}

For $L$ a suplattice, the (contravariant) hom-functor $\mathbf{SL}(\__ ,L)\,:\,\mathbf{SL}^{op}\to\mathbf{Set}\ $ can  be also regarded as a functor $$\mathbf{SL}(\__ ,L)\,:\,\mathbf{SL}\to\mathbf{SL}^{op}\ $$ where, for $f\in\mathbf{SL}(X,Y)$ and $\varphi\in\mathbf{SL}(Y,L)$, we have $\mathbf{SL}(f,L)(\varphi)$ = $\varphi\circ f$.
%; this makes sense because $$\varphi\circ (\bigvee_{i\in I}f_i)=\bigvee_{i\in I}(\varphi\circ f_i)$$
%for every family $(f_i)_{i\in I}$ of elements of $\mathbf{SL}(X,Y)$.

Another well-known contravariant endofunctor is the subobject functor
$$\mathbf{Sub}\,:\,\mathbf{SL}^{op}\to\mathbf{PreOrd}$$
which sends each suplattice $L$ to the preorder $\mathbf{Sub}(L)$ of subobjects of $L$ in $\mathbf{SL}$. Recall that a suboject of $L$ can be represented as a subset $I\subseteq L$ closed under joins in $L$. Given $f:M\rightarrow L$ in $\mathbf{SL}$ and $I\in\mathbf{Sub}(L)$, $\mathbf{Sub}(f)$ sends $I$ to the pullback $\{x\in M\ |\ f(x)\in I\}$ of $I$ along $f$.
$$\xymatrix@M=9pt{
\mathbf{Sub}(f)(I)\ar@{>->}[d]\ar[r]\pullbackcorner				& I\ar@{>->}[d]^{i}\\
M\ar[r]_{f}		&L\\}$$
The composition $\mathbf{Sub}\circ \mathbf{SL}(\__,\Omega)$ is a functor $$\mathbf{P}:\mathbf{SL}\rightarrow \mathbf{PreOrd}$$
which, of course, can also be read as a contravariant functor on $\mathbf{SL}^{op}$
$$\mathbf{P}:(\mathbf{SL}^{op})^{op}\rightarrow \mathbf{PreOrd},$$
that is, a doctrine on $\mathbf{SL}^{op}$.
By the so-called Grothendieck construction \cite{BJ}, we get a category $\int\mathbf{P}$ whose objects are pairs 
$(L,\Phi)$ with $L$ a suplattice and $\Phi$ a subobject of $\mathbf{SL}(L,\Omega)$ in $\mathbf{SL}$.
An arrow $(L,\Phi)\to(M,\Psi)$ in $\int\mathbf{P}$ is a suplattice homomorphism $f:M\rightarrow L$ such that 
$$\Phi\subseteq\mathbf{P}(f)(\Psi).$$
Since $\mathbf{P}(f)(\Psi)$ = $\{\varphi\in\mathbf{SL}(L,\Omega)\ |\ \varphi\circ f\in\Psi\}$ by definition, such a condition is equivalent to the following
$$\Phi\circ f\subseteq \Psi,$$
where $\Phi\circ f:=\{\varphi\circ f \mid \varphi\in \Phi\}$. Therefore, $\int\mathbf{P}$ is exactly the category $\mathbf{BTop}$ of Sambin's basic topologies \cite{bp}, which we introduced in Section \ref{sec:BTop} above.

This construction yields a forgetful functor $ \mathbf{U}:\int \mathbf{P}\rightarrow \mathbf{SL}^{op}$, which is in fact a fibration. The functor has a right adjoint, the \emph{constant object functor} 
$$\mathbf{\Delta}: \mathbf{SL}^{op}\rightarrow \int\mathbf{P},$$ 
which sends each suplattice $L$ to the object $(L,\mathbf{SL}(L,\Omega))$ and each $f:L\rightarrow M$ in $\mathbf{SL}^{op}$ to itself as an arrow from $\mathbf{\Delta(L)}$ to $\mathbf{\Delta(M)}$ in $\int \mathbf{P}$.

%The functor $\Delta$ is right adjoint to $ U$: for every $L$ in $\mathbf{SL}^{op}$ and $(L',\Phi)$ in $\int\mathbf{P}$, $f$ is an arrow in $\int\mathbf{P}$ from $(L',\Phi)$ to $\Delta(L)$ if and only if $f$ is an arrow in $\mathbf{SL}$ from $L$ to $L'= U(L',\Phi)$ and $\Phi\circ f\subseteq \mathbf{SL}(L,\Omega)$; however the second condition always holds, hence this is equivalent to require that $f$ is an arrow from $ U(L',\Phi)$ to $L$ in $\mathbf{SL}^{op}$. 

$$\mathbf{BTop}\ =
\xymatrix{
 \int\mathbf{P}\ar@/^0.5pc/[rr]^-{ \mathbf{U}}	&\bot	&\mathbf{SL}^{op}\ar@/^0.5pc/[ll]^-{\mathbf{\Delta}}\\
}$$
Moreover $ \mathbf{U}\circ\mathbf{\Delta}$ is just the identity functor on $\mathbf{SL}^{op}$. Thus, $\mathbf{\Delta}$ is full and faithful, and so $\mathbf{SL}^{op}$ can be regarded as a reflective subcategory of $\int \mathbf{P}$. In this way, we recover the result in \cite{galois}.

Note that the monad $T$  induced by the adjunction $ \mathbf{U}\dashv\mathbf{\Delta}$ is an idempotent monad. By the results in Section $4.2$ of \cite{borceux2}, we have that ${\bf SL}^{op}$ is equivalent both to the category of free algebras (the Kleisli category) and to the category of algebras (the Eilenberg--Moore category) on $T$. Hence the adjunction $ \mathbf{U}\dashv\mathbf{\Delta}$ is monadic.

\paragraph{Remark.}

In a suplattice, arbitrary meets always exist, that is, if $(L,\leq)$ is a suplattice, then $(L,\leq)^{op} := (L,\geq)$ is a suplattice as well. Moreover, every suplattice homomorphism $f$ has a right adjoint $f^{op}$ which preserves all meets. This determines  a contravariant functor $(\_)^{op}$, which is in fact a self-duality of $\mathbf{SL}$, and 
$$\mathbf{SL}(X,Y)\cong\mathbf{SL}(Y^{op},X^{op})$$
for all $X$ and $Y$. 
%In particular, $\mathbf{SL}$ is isomorphic to $\mathbf{SL}^{op}$ and one can easily notice that an arrow $f$ is a mono in $\mathbf{SL}$ if and only if its right adjoint $f^{op}$ is an epi in $\mathbf{SL}$ if and only if $f^{op}$ is a mono in $\mathbf{SL}^{op}$; hence $\mathbf{Sub}_{\mathbf{SL}}(L^{op})\simeq \mathbf{Sub}_{\mathbf{SL}^{op}}(L)$ for every suplattice $L$.

Classically, $\mathbf{SL}(\__,\Omega)$ is naturally isomorphic to the functor $(\_)^{op}$ because $\Omega^{op}\cong\Omega$ so that $\mathbf{SL}(L,\Omega)$ $\cong$ $\mathbf{SL}(\Omega,L^{op})$ $\cong$ $L^{op}$.\footnote{This cannot hold constructively, for if $\varphi: \Omega^{op}\to\Omega$ were an isomorphism, then we could prove $\neg\neg p\leq p$ for every $p\in\Omega$ as follows. Indeed $p=0$ implies $\neg\neg p=0$, and so $\varphi(p)=1$ implies $\varphi(\neg\neg p)=1$. By a characteristic feature of $\Omega$, this gives $\varphi(p)\leq\varphi(\neg\neg p)$, hence $\neg\neg p\leq p$.} 
%So {\bf SL} is closed, actually it is a closed symmetric monoidal category. The unit of the tensor is $\Omega$, the suplattice of all subsets of a singleton set (classically, $\Omega\cong\{0,1\}$). For every suplattice $L$, we have that $\mathbf{SL}(\Omega,L)$ $\cong$ $L$ and that  $\mathbf{SL}(L,\Omega^{op})\cong\mathbf{SL}(\Omega,L^{op})\cong L^{op}$. Therefore $\mathbf{SL}(\mathbf{SL}(L,\Omega^{op}),\Omega^{op})$ $\cong$ $\mathbf{SL}(L^{op},\Omega^{op})$ $\cong$ ${L^{op}}^{op}$ $\cong$ $L$, and hence {\bf SL} is even a star-autonomous category.
%Constructively, we have a natural transformation $Id_\mathbf{SL}\rightarrow\mathbf{SL}(\mathbf{SL}(\_,\Omega),\Omega)$ whose component at $L$ is the suplattice homomorphism  $$\begin{array}{rcl} L & \rightarrow & \mathbf{SL}(\mathbf{SL}(L,\Omega),\Omega)\\ x & \mapsto & [\varphi\mapsto \varphi(x)] \end{array}$$ which apparently is not a monomorphism neither an epimorphism.
%
Therefore, for every $L$, $\mathbf{P}(L)=\mathbf{Sub}(\mathbf{SL}(L,\Omega))\cong  \mathbf{Sub}(L^{op})$ which is isomorphic to the lattice of all suplattice quotients of $L$. In other words, an object $(L,\Phi)$ corresponds to an epimorphism $e:L\to\Phi^{op}$, and an arrow $(L,\Phi)\to(M,\Psi)$ is a suplattice homomorphism $f:M\to L$ such that $e\circ f: M\to\Phi^{op}$ preserves the congruence relation on $M$ corresponding to $\Psi$.

%\end{remark}

\subsection{The case of frames (and locales)} \label{sec:DoctLoc}

The category $\mathbf{Frm}$ of frames is the subcategory of $\mathbf{SL}$ whose objects are frames and whose arrows preserve finite meets (in addition to arbitrary joins). 
The category $\mathbf{Loc}$ of locales is defined as $\mathbf{Frm}^{op}$.
By restricting the functor $\mathbf{P}$ to frames, we get a doctrine
$$\widetilde{\mathbf{P}}:\mathbf{Loc}^{op}=\mathbf{Frm}\longrightarrow\mathbf{PreOrd}$$ 
on $\mathbf{Loc}$,
which gives rise to a fibration
$ \mathbf{U}:\int\widetilde{\mathbf{P}}\rightarrow \mathbf{Loc}$ fitting in a pullback square in $\mathbf{Cat}$ as follows.

$$\xymatrix{
 \int\widetilde{\mathbf{P}}\ \ar@{>->}[r]\ar[d]_{ \mathbf{U}}\pullbackcorner	&\int\mathbf{P}\ar[d]^{ \mathbf{U}}\\
\mathbf{Loc}\ \ar@{>->}[r]			&\mathbf{SL}^{op}\\
}$$
Here $\int\widetilde{\mathbf{P}}$ is exactly the category $\mathbf{PTop}$ as introduced in Section \ref{sec:PTop}.

\section{Weakly sober spaces}
%\subsection{Topological spaces and closed subsets}

\subsection{Irreducible closed subsets}\label{irreducible}

The open sets of a  topological space $(X,\tau)$ form a frame with respect to the set-theoretic unions and intersections. 
A subset $C\subseteq X$ is {\bf closed} if
$$(\forall I\in\tau)(x\in I\ \Rightarrow\ C\between I)\ \Longrightarrow x\in C$$
for all $x\in X$. The collection $\mathsf{Closed}(X,\tau)$ of closed subsets in $(X,\tau)$ is a complete lattice (where infima are given by intersections, and joins are given by closure of unions), but constructively need not be a co-frame.\footnote{For a Brouwerian counterexample consider the discrete space and recall that the so-called ``constant domain axiom" $\forall x(\varphi\vee\psi)\to\varphi\vee\forall x\,\psi$, with $x$ not free in $\varphi$, is not provable constructively.} 
%Without LEM, a closed set $C$ need not be the (pseudo)complement $X\setminus A$ = $\{x\in X\ |\ \neg(x\in A)\}$ of some open set $A$,\footnote{In the discrete case $\tau=\mathcal{P}(X)$, every subset is closed (to see this, let $A$ in \eqref{eq.closed} vary over singletons) but need not be ``stable".} although the converse is true, that is, $X\setminus A$ is closed whenever $A$ is open. So the mapping $$\begin{array}{rcl}\tau^{op} & \longrightarrow & \mathsf{Closed}(X,\tau) \\ A & \longmapsto & X\setminus A\end{array}$$ is one-to-one and preserves joins, so it is a monomorphism in {\bf SL}. 

As usual, it makes sense to define the closure $\mathsf{cl}D$ of a subset $D\subseteq X$ as the intersection of all closed subsets containing $D$. 

Every closed subset $C$ of $X$ determines a map 
$$
\begin{array}{rcrcl}
\varphi_C & : & \tau & \longrightarrow & \Omega \\
 & & I &\longmapsto & C\between I
\end{array}
$$ 
which preserves joins, that is, $\varphi_C\in\mathbf{SL}(\tau,\Omega)$.
Note that $\varphi_D$ makes sense also when $D$ is an arbitrary subset; however $\varphi_{D}=\varphi_{\mathsf{cl}D}$ because $I\between D$ if and only if $I\between\mathsf{cl}D$ for every $I\in \tau$. So the mapping 
$$\begin{array}{rcl}
\mathsf{Closed}(X,\tau) & \longrightarrow & \mathbf{SL}(\tau,\Omega) \\
 C & \longmapsto & \varphi_C
\end{array}$$
is injective and preserves joins. Thus $\mathsf{Closed}(X,\tau)$ is a subobject of ${\bf SL}(\tau,\Omega)$.\footnote{Classically, every $\varphi\in\mathbf{SL}(\tau,\Omega)$ is of the form $\varphi_{C}$: take $C$ to be the closed subset  \mbox{$X\setminus\bigcup\left\{I \in \tau \mid \varphi(I)=0 \right\}$}. Hence $\mathsf{Closed}(X,\tau)$ $\cong$ ${\bf SL}(\tau,\Omega)$. This cannot be the case constructively, as we will see in Section \ref{weak-sobriety}.}

A closed subset $C\subseteq X$ is {\bf irreducible} if any of the following equivalent conditions holds:
\begin{enumerate}
\item $\varphi_C$ preserves finite meets;
\item $C$ is inhabited and for every $I,J\in\tau$, if $I\between C$ and $J\between C$, then $(I\cap J)\between C$;
\item $\{I\in\tau\ |\ I\between C\}$ is a completely-prime filter of opens.
\end{enumerate}
In other words, a closed subset $C$ is irreducible if and only if $\varphi_C$ is a frame homomorphism, that is, a \emph{point} in the sense of locale theory. 
However we cannot show constructively that all frame homomorphisms $\tau\to\Omega$ arise in this way; see Section \ref{weak-sobriety}. 

Classically, $C$ is irreducible if and only if it is non-empty and cannot be written as a disjoint union of two non-empty closed subsets \cite{stone_spaces}. Therefore, $\{C\subseteq X\ |\ C\textrm{   irreducible closed}\}$ can be identified with $\mathbf{Frm}(\tau,\Omega)$.

\subsection{Weak sobriety}\label{weak-sobriety}

Recall that a space is $T0$ or Kolmogorov if $x=y$ follows from the assumption that $\mathsf{cl}\{x\}=\mathsf{cl}\{y\}$. Since $\mathsf{cl}\{x\}$ is always irreducible, we have the following embedding for a $T0$ space $(X,\tau)$:
$$X\hookrightarrow\{C\subseteq X\ |\ C\textrm{   irreducible closed}\}\hookrightarrow\mathbf{Frm}(\tau,\Omega).$$
A $T0$ space is {\bf weakly sober} if every irreducible closed subset is the closure of a singleton, that is, if the embedding $X\hookrightarrow\{C\sub X\ |\ C\textrm{ irreducible closed}\}$ is an isomorphism. It is {\bf sober} if the embedding $X\hookrightarrow\mathbf{Frm}(\tau,\Omega)$ is an isomorphism.
%\footnote{This means that for every $\varphi\in\mathbf{Frm}(\tau,\Omega)$ there is a corresponding $x\in X$ such that, for every $A\in\tau$, $\varphi(A)$ is the truth value of $x\in A$.}
Note that every weakly sober space is sober classically.

Constructively, every Hausdorff space is weakly sober. However, if every weakly sober space were sober, the non-constructive principle LPO (the Limited Principle of Omniscience) would be derivable \cite{aczel-fox}. 
Thus, we cannot prove that all $\varphi\in\mathbf{SL}(\tau,\Omega)$ are of the form $\varphi_C$ for some closed subset $C$; otherwise $\mathbf{Frm}(\tau,\Omega)$ could be identified with the irreducible closed subsets, which would make sobriety and weak sobriety coincide.

\section{Factorizing the $\mathbf{Top}$--$\mathbf{Loc}$ adjunction}

The usual $\mathbf{\Omega}\dashv \mathbf{Pt}$ adjunction between the category  $\mathbf{Top}$ of topological spaces and the category $\mathbf{Loc}$ of locales does not compose with the adjunction $\mathbf{U}\dashv\mathbf{\Delta}$ between $\mathbf{Loc}$ and $\mathbf{PTop}$ to give an adjunction between $\mathbf{Top}$ and $\mathbf{PTop}$. 

$$\xymatrix{
\mathbf{Top} \ar@/^/^{\mathbf{\Omega}}[rr] &\bot& \mathbf{Loc}\ar@/^/^{\mathbf{Pt}}[ll] \ar@/_/_{\mathbf{\Delta}}[rr] &\,\,\bot& \mathbf{PTop}\ar@/_/_{\mathbf{U}}[ll]
}$$

Nevertheless, a meaningful adjunction between $\mathbf{Top}$ and $\mathbf{PTop}$ can be given, as explained in the following, through which the usual $\mathbf{Top}$--$\mathbf{Loc}$ adjunction factors. 

\subsection{Points of a positive topology}

The suplattice $\Omega$ is an initial frame, that is, a terminal locale. Hence $\mathbf{\Delta}(\Omega)$ is a terminal object in $\mathbf{PTop}$. We define a {\bf point} of a positive topology $(L,\Phi)$ as a global point  $\mathbf{\Delta}(\Omega)\to(L,\Phi)$ in $\mathbf{PTop}$, and we write $\mathbf{Pt}^{+}(L,\Phi)$ instead of $\mathbf{PTop}(\mathbf{\Delta}(\Omega),(L,\Phi))$.  Thus, a point of $(L,\Phi)$ is a frame homomorphism $f:L\to\Omega$ such that $\mathbf{SL}(\Omega,\Omega)\circ f\subseteq\Phi$. Since $\mathbf{SL}(\Omega,\Omega)$ contains the identity map, we have  $f\in\Phi$. Conversely, if $f\in \Phi$ and $\varphi\in\mathbf{SL}(\Omega,\Omega)$, then we have $\varphi\circ f=\bigvee\{x\in \{f\}|\,\varphi=\mathsf{id}_{\Omega}\}\in \Phi$. In other words, the points of $(L,\Phi)$ are exactly those points of the locale $L$ that are in $\Phi$. Hence, $\mathbf{Pt}^+(L,\Phi)$ can be regarded as a subspace of the topological space $\mathbf{Pt}(L)$. 

The construction $\mathbf{Pt}^{+}$ can be extended to a functor from $\mathbf{PTop}$ to $\mathbf{Top}$ as follows. Given an arrow $(L,\Phi)\to(M,\Psi)$ with underlying frame homomorphism $f:M\to L$, the continuous map $\mathbf{Pt}(f):\mathbf{Pt}(L)\to \mathbf{Pt}(M)$, which sends a point $p:L\to\Omega$ to the point $p\circ f:M\to\Omega$, can be restricted to a continuous map $\mathbf{Pt}^+(L,\Phi)\to \mathbf{Pt}^+(M,\Psi)$ because $\Phi\circ f\subseteq\Psi$.

%It follows that $\mathbf{Pt}^{+}$ is a functor from $\mathbf{PTop}$ to $\mathbf{Top}$.

\subsection{The canonical positive topology associated with a space}\label{sec:CanPos}

As shown in Section \ref{irreducible}, the closed subsets $\mathsf{Closed}(X,\tau)$ of a topological space $(X,\tau)$ can be seen as a sub-suplattice of $\mathbf{SL}(\tau,\Omega)$ via the mapping \mbox{$C\mapsto\varphi_C$}. 
Thus, we can define a functor $\mathbf{\Lambda}:\mathbf{Top}\to\mathbf{PTop}$ whose object part is $$\mathbf{\Lambda}(X,\tau)\ =\ \big(\tau,\{\varphi_C \mid C \textrm{ is closed}\}\big).$$ For a continuous map $f:(X,\tau_X)\to (Y,\tau_Y)$, the $\mathbf{PTop}$-morphism $\mathbf{\Lambda}(f)$ is just the locale morphism corresponding to the frame homomorphism \mbox{$f^{-1}:\tau_Y\to\tau_X$}. This makes sense because for any closed subset $C\subseteq X$, the suplattice homomorphism $\varphi_C\circ f^{-1}:\tau_Y\to\Omega$ is precisely  $\varphi_D$, where $D=\mathsf{cl}f(C)$.

\subsection{The adjunction between $\mathbf{Pt}^+$ and $\mathbf{\Lambda}$}

\begin{theorem*}The following hold:
\begin{enumerate}
\item $\mathbf{Pt}=\mathbf{Pt}^+\circ \mathbf{\Delta}$;
\item $\mathbf{\Omega}= \mathbf{U}\circ \mathbf{\Lambda}$;
\item $\mathbf{\Lambda}\dashv \mathbf{Pt}^+$.
\end{enumerate}
As a consequence, the adjunction between $\mathbf{Top}$ and $\mathbf{Loc}$ factors through an adjunction between $\mathbf{PTop}$ and $\mathbf{Loc}$. 
$$
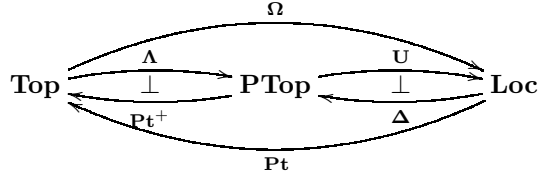
\xymatrix{
\mathbf{Top}\ar@/^2pc/[rrrr]^-{\mathbf{\Omega}}\ar@/^0.5pc/[rr]^-{\mathbf{\Lambda}}	&\bot	&\mathbf{PTop}\ar@/^0.5pc/[ll]^-{\mathbf{Pt}^+}\ar@/^0.5pc/[rr]^-{ \mathbf{U}}	&\bot	&\mathbf{Loc}\ar@/^0.5pc/[ll]^-{\mathbf{\Delta}}\ar@/^2pc/[llll]^-{\mathbf{Pt}}\\
}$$
\end{theorem*}
\begin{proof}
For every locale $L$, $\mathbf{Pt}(L)=\mathbf{Pt}(L)\cap \mathbf{SL}(L,\Omega)=\mathbf{Pt}^+(\mathbf{\Delta}(L))$, and for every topological space $(X,\tau)$, $ \mathbf{U}(\mathbf{\Lambda}(X,\tau))=\tau=\mathbf{\Omega}(X,\tau)$. Hence 1 and 2 hold.

For 3, if $f:\mathbf{\Lambda}(X,\tau)\to (L,\Phi)$, then one can define a continuous function $\widetilde{f}$ from $(X,\tau)$ to $\mathbf{Pt}^+(L,\Phi)$ as follows:
$$\widetilde{f}(x):=\varphi_{\mathsf{cl}\{x\}}\circ f,$$
that is, for every $y\in L$, $\widetilde{f}(x)(y):=\mathsf{cl}\{x\}\between f(y)\in \Omega$.

Conversely, if $g$ is a continuous function from $(X,\tau)$ to $\mathbf{Pt}^+(L,\Phi)$, then an arrow  $\widehat{g}$ from $\mathbf{\Lambda}(X,\tau)$ to $(L,\Phi)$ is defined as follows:
$$\widehat{g}(y):=g^{-1}(\{\varphi\in \mathbf{Pt}(Y)\cap \Psi|\,\varphi(y)=1\})\in \tau$$
for every $y\in L$. This is an arrow in $\mathbf{PTop}$ because for every closed subset $C\sub X$, we have $\varphi_C\circ\widehat{g}$ = $\bigvee\{\varphi\in\mathbf{Pt}^+(L,\Phi)\ |\ \varphi\in g(C)\}\in\Phi$.

The maps $\widetilde{(\_)}$ and $\widehat{(\_)}$ define a natural isomorphism between the functors $\mathbf{PTop}(\mathbf{\Lambda}(\_),\_)$ and $\mathbf{Top}(\_\,,\mathbf{Pt}^+(\_))$.
\end{proof}

A topological space $(X,\tau)$ is weakly sober when $(X,\tau)\cong \mathbf{Pt}^+(\mathbf{\Lambda}(X,\tau))$, while it is sober when $(X,\tau)\cong \mathbf{Pt}(\mathbf{\Omega}(X,\tau))$. 

Classically, $\mathbf{SL}(\tau,\Omega)$ = $\{\varphi_C\ |\ C \textrm{ is closed}\}$ holds. Hence  $\mathbf{\Lambda}=\mathbf{\Delta}\circ\mathbf{\Omega}$, and thus $\mathbf{Pt}^+\circ\mathbf{\Lambda}=\mathbf{Pt}^+\circ \mathbf{\Delta}\circ\mathbf{\Omega}=\mathbf{Pt}\circ\mathbf{\Omega}$. Therefore, as already noted, sobriety and weak sobriety coincide classically. 

\section{Limits and colimits in {\bf BTop} and {\bf PTop}}

Whenever the base $\mathbb{C}$ of a doctrine $\mathbf{P}:\mathbb{C}^{op}\rightarrow \mathbf{PreOrd}$ is complete and the same holds for every fiber $\mathbf{P}(A)$, the Grothendieck construction $\int \mathbf{P}$ gives a complete category.
If $(L_{i},\Phi_{i})_{i\in I}$ is a set-indexed family of objects in $\int \mathbf{P}$, its product is given by the object 
$$\prod_{i\in I}(L_{i},\Phi_{i})\ =\ \left(\prod_{i\in I}L_{i},\bigwedge_{i\in I}\mathbf{P}(\pi_{i})(\Phi_{i})\right)$$
together with the projections $\pi_{i}$ inherited from $\mathbb{C}$. The equalizer of two parallel arrows  $f,g:(L,\Phi)\rightarrow (M,\Psi)$ in $\int \mathbf{P}$ is $e:\big(E,\mathbf{P}(e)(\Phi)\big)\rightarrow(L,\Phi)$, where $e:E\rightarrow L$ is the equalizer of $f$ and $g$ in $\mathbb{C}$.

If, moreover, $\mathbb{C}$ is cocomplete and, for every arrow $f$ of $\mathbb{C}$, the monotone map $\mathbf{P}(f)$ has a left adjoint $\exists_f$, then $\int\mathbf{P}$ is cocomplete as well. The coproduct of a family of objects $(L_{i},\Phi_{i})_{i\in I}$ in $\int \mathbf{P}$ is given by the object $$\sum_{i\in I}(L_i,\Phi_i)\ = \left(\sum_{i\in I}L_{i},\bigvee_{i\in I}\exists_{j_{i}}(\Phi_{i}) \right)$$
together with the injections $j_{i}$ inherited from $\mathbb{C}$. Finally, the coequalizer of  $f,g:(L,\Phi)\rightarrow (M,\Psi)$ is $q:(M,\Psi)\rightarrow(Q,\exists_{q}(\Psi))$, where $q:M\rightarrow Q$ is the coequalizer of $f$ and $g$ in $\mathbb{C}$.

 %Both in the case of $\mathbf{BTop}$ and $\mathbf{PTop}$ these requirements are satisfied.
 The doctrines $\mathbf{P}$ and $\widetilde{\mathbf{P}}$ introduced in Section \ref{sec:DoctSL} and Section \ref{sec:DoctLoc}, respectively, satisfy the above requirements.
 Indeed, every fiber of $\mathbf{P}$ and $\widetilde{\mathbf{P}}$ is a complete lattice because an arbitrary intersection of sub-suplattices is a sub-suplattice. Moreover, it is well known that both $\mathbf{SL}^{op}$ and $\mathbf{Loc}$ are complete and cocomplete. Finally, every $\mathbf{P}(f)$ has a left adjoint, namely $\exists_f(\Phi)=\Phi\circ f$, essentially by the very definition of $\mathbf{P}$. Hence, the categories $\mathbf{PTop}$ and $\mathbf{BTop}$ are complete and cocomplete.
 
\section*{Acknowledgements}
 Part of this work was carried out while the second author was visiting University of Padova in December 2018. The visit was supported by Core-to-Core Program (A. Advanced Research Networks) of Japan Society for the Promotion of Science.
\bibliographystyle{abbrv}
\bibliography{biblioCKM}

\end{document}